\documentclass[11pt,oneside]{amsart}
\usepackage[utf8]{inputenc}
\usepackage{geometry,amssymb,amsthm,amsmath,
graphics,enumerate}

\usepackage{caption}

\usepackage{times}
\usepackage{amsfonts}
\usepackage{amssymb}
\usepackage{amscd}
\usepackage{url}
\usepackage{graphicx}

\usepackage[cmtip]{xy}

\usepackage[linecolor=black,prependcaption,textsize=small]{todonotes}

\usepackage{xargs}

\newfont{\msbm}{msbm10 scaled\magstephalf}
\def\aut{\rm aut}

\def\rg{\mathop{\rm rg}}

\Alph{thrmA}
 \newtheorem{thrm}{Theorem}[section]
\newtheorem{theorem}[thrm]{Theorem}
\newtheorem{assumption}[thrm]{Assumption}

\newtheorem{lemma}[thrm]{Lemma}

\newtheorem{corollary}[thrm]{Corollary}
\newtheorem{remark}[thrm]{Remark}

\newtheorem{notation}[thrm]{Notation}

\newtheorem{warning}[thrm]{Warning}

\newtheorem{fact}[thrm]{Fact}

\newtheorem{definition}[thrm]{Definition}
\newtheorem{example}[thrm]{Example}

\theoremstyle{plain}
\newtheorem{thmx}{Theorem}

\newcommand{\acl}{\mathrm{acl}}
\newcommand{\tp}{\mathrm{tp}}
\newcommand{\qtp}{\mathrm{qtp}}

\def\b1K{\mbox{\boldmath  K }_{-1}}

\def\bK{\mbox{\boldmath  K }}

\def\MMM{\mathop{  \mathbb{M}}}

\def\dom{\unrhd}

\def\th{\mathop{\rm Th}}

\def\qftp{\mathop{\bf qftp}}

\def\<{\langle}

\def\>{\rangle}

 \def\dom{\mathop{\rm dom}}

\def\psl{\mathop{\rm PSL}}

\def\comm{\mathop{\rm Comm}}

\def\dcl{{\rm dcl}}

\def\acl{{\rm acl}}

\def\cl{{\rm cl}}

\newbox\noforkbox \newdimen\forklinewidth
\setbox1\hbox to \wd0{\hfil\vrule width \forklinewidth depth-2pt
 height 10pt \hfil}
\wd1=0 cm \setbox\noforkbox\hbox{\lower 2pt\box1\lower
2pt\box0\relax}
\def\unionstick{\mathop{\copy\noforkbox}\limits}
\def\nonfork_#1{\unionstick_{\textstyle #1}}
\newbox\doesforkbox
\setbox\doesforkbox\hbox{\lower 2pt\box1 \lower
2pt\box2\lower2pt\box0\relax}

\def\aut{\rm aut}

\newcommand{\forkindep}[1][]{%
  \mathrel{
    \mathop{
      \vcenter{
        \hbox{\oalign{\noalign{\kern-.3ex}\hfil$\vert$\hfil\cr
              \noalign{\kern-.7ex}
              $\smile$\cr\noalign{\kern-.3ex}}}
      }
    }\displaylimits_{#1}
  }
}
\newcommand{\nonforkindep}[1][]{%
  \mathrel{
    \mathop{
      \vcenter{
        \hbox{\oalign{\noalign{\kern-.3ex}\hfil$\vert$\rlap{$'$}\hfil\cr
              \noalign{\kern-.7ex}
              $\smile$\cr\noalign{\kern-.3ex}}}
      }
    }\displaylimits_{#1}
  }
}

\newcommand{\HH}{\mbox{\msbm H}}
\newcommand{\RR}{\mbox{\msbm R}}

\newcommand{\QQ}{\mbox{\msbm Q}}

\newcommand{\MM}{\mbox{\msbm M}}
\newcommand{\CC}{\mbox{\msbm C}}

\def\sub'm{\prec_{\bK'}}
\def\grpf #1 #2{{\rm grp}_{#2}(#1)}
\def\spanf #1 #2{{\rm span}_{#2}(#1)}
\def\fldf #1 #2{{\rm fld}_{#2}(#1)}
\def\dclf #1 #2{{\rm dcl}_{#2}(#1)}
\def\rclf #1 #2{{\rm rcl}_{#2}(#1)}
\def\aclf #1 #2{{\rm acl}_{#2}(#1)}
\def\acff #1 #2{{\rm acf}_{#2}(#1)}
\def\strf #1 #2{{\rm str}_{#2}(#1)}
\def\tclf #1 #2{{\rm acf}_{#2}(#1)}

\def\abar{{\bf a}}
\def\bbar{{\bf b}}

\def\dbar{{\bf d}}
\def\ebar{{\bf e}}
\def\fbar{{\bf f}}
\def\gbar{{\bf g}}
\def\hbar{{\bf h}}

\def\jbar{{\bf j}}

\def\ubar{{\bf u}}
\def\vbar{{\bf v}}
\def\wbar{{\bf w}}
\def\xbar{{\bf x}}
\def\ybar{{\bf y}}
\def\zbar{{\bf z}}
 
\def\tp{{\rm tp }}

\newcommand{\sidebar}[1]{\vskip10pt\noindent
 \hskip.70truein\vrule width2.0pt\hskip.5em
 \vbox{\hsize= 4truein\noindent\footnotesize\relax #1 }\vskip10pt\noindent}
\date{\today}

\begin{document}
\author{John T. Baldwin}
\address{John T. Baldwin\\ Department of Mathematics, Statistics, and Computer Science\\
	University of Illinois Chicago\\
	322 Science and Engineering Offices (M/C 249)\\
	851 S. Morgan Street\\
	Chicago, IL 60607-7045}
\email{jbaldwin@uic.edu}
\author{Joel Nagloo}
\address{Joel Nagloo\\ Department of Mathematics, Statistics, and Computer Science\\
	University of Illinois Chicago\\
	322 Science and Engineering Offices (M/C 249)\\
	851 S. Morgan Street\\
	Chicago, IL 60607-7045}
\email{jnagloo@uic.edu}

\title{Categoricity and non-arithmetic Fuchsian groups} 

\thanks{J. Nagloo is partially supported by NSF grant DMS-2348885. Some of the work in this paper was completed while J. Nagloo was supported by an AMS Centennial Fellowship.}

\begin{abstract}
Let $\Gamma \subset PSL_2(\mathbb{R})$ be a non-arithmetic Fuchsian group of the first kind with finite covolume, and let $j_{\Gamma}$ be a corresponding uniformizer. In this paper we introduce a natural $L_{\omega_1,\omega}$-axiomatization $T^{\infty}_{SF}$ of the theory of $j_{\Gamma}$ viewed as a covering map. We show that $T^{\infty}_{SF}$ is categorical in all infinite cardinalities, extending to the non-arithmetic setting earlier results of Daw and Harris obtained in the arithmetic case. We also show that the associated first-order theory $T_{j_{\Gamma}}$ is complete, admits elimination of quantifiers, and is $\omega$-stable.
\end{abstract}

\maketitle


\section{Introduction}

Let $\Gamma\subset PSL_2(\mathbb{R})$ be a Fuchsian group of the first kind and of finite covolume, and let $j_{\Gamma}:\mathbb{H}_{\Gamma}\rightarrow S_{\Gamma}(\mathbb{C})$ be a uniformizer, where $\mathbb{H}_{\Gamma}:=\mathbb{H}\cup C_{\Gamma}$ and $C_{\Gamma}$ is the set of cusps of $\Gamma$. Thus, $j_{\Gamma}$ is a $\Gamma$-automorphic function: a covering map realizing the smooth projective curve $S_{\Gamma}(\mathbb{C})$ as the quotient $\Gamma\setminus\mathbb{H}_{\Gamma}$.  This classical situation, originating in the work of Poincar\'e and Klein at the turn of the twentieth century, lies at the intersection of complex analysis, hyperbolic geometry, and the theory of automorphic functions.

This paper draws on two streams of work
 studying the function $j_{\Gamma}$ through the lens of model theory. On the one hand, there is the work of Bl\'azquez-Sanz, Casale, Freitag and Nagloo (cf. \cite{CFN}, \cite{BCFN}),  which provides a full classification of the structure of the set $\mathcal{X}_{\Gamma}$ defined — in a differentially closed field — by the Schwarzian differential equation satisfied by $j_{\Gamma}$. Their work shows that a key dividing line in geometric group theory, namely the arithmeticity of the group $\Gamma$, fully governs the complexity of algebraic relations among the solutions in $\mathcal{X}_{\Gamma}$. Roughly speaking, they show that the Fuchsian group $\Gamma$ is arithmetic if and only if $\mathcal{X}_{\Gamma}$ is not $\omega$-categorical, that is, if and only if the solution set admits infinitely many `algebraic symmetries'. It is a fundamental result of Margulis that the existence of infinitely many `Hecke' correspondences (equivalently, a large commensurator $\text{Comm}(\Gamma)$) characterizes the arithmeticity of the group. In particular, one of the main achievements of the work of Bl\'azquez-Sanz, Casale, Freitag and Nagloo is to show that these `Hecke' correspondences account for {\em all} algebraic relations in $\mathcal{X}_{\Gamma}$.

The second line of work in model theory studying $j_{\Gamma}$ is also concerned with questions of categoricity,  but in uncountable cardinalities rather than $\aleph_0$. Namely, one may view $j_{\Gamma}:\mathbb{H}_{\Gamma}\rightarrow S_{\Gamma}(\mathbb{C})$ as a two-sorted structure $(\mathbb{H}_{\Gamma};S_{\Gamma}(\mathbb{C}))$ in a natural language and find an $L_{\omega_1,\omega}$-axiomatization of the theory $T_{j_{\Gamma}}$ of that structure. This perspective is rooted in Zilber’s  program to find axiomatizations of canonical structures that are categorical in all uncountable powers. This gives an `algebraic'  characterization of structures with inherent analytic properties. A natural question is whether $T_{j_{\Gamma}}$  exemplifies this program — in particular, whether the unique model of cardinality continuum must coincide with the classical analytic structure. Daw and Harris \cite{DawHarris} showed that if $\Gamma$ is arithmetic, then $T_{j_{\Gamma}}$ is indeed categorical in all uncountable powers. This result was later generalized by Eterovi\'c \cite{Etrev} to include arithmetic uniformizations in higher dimensions. Their work makes essential use of deep results from arithmetic geometry, in particular open image theorems for Galois representations governing Hecke orbits on Shimura varieties, while also showing that categoricity itself can force arithmetic statements  about the Galois representations associated with points on the Shimura variety via a theorem of Keisler. In higher dimensions, this relies on open image conditions that are presently conjectural and closely related to the Mumford–Tate conjecture.

It is natural to ask whether the arithmeticity of the group $\Gamma$ again provides a dividing line for categoricity in this new sense. In this paper, we show that, unlike in the context of differentially closed fields, the categoricity of $T_{j_{\Gamma}}$ is not governed by the arithmetic nature of $\Gamma$. More precisely, we prove that when $\Gamma$ is non-arithmetic, the theory $T_{j_{\Gamma}}$ is categorical in all uncountable cardinalities. Moreover, the proof in the non-arithmetic case is substantially simpler than in the arithmetic setting, and its structure clarifies the respective roles played by model-theoretic and geometric considerations in the earlier results.

Our argument follows the same pattern as \cite{DawHarris}: i) axiomatize the class of models in $L_{\omega_1,\omega}$ (which we named $T^{\infty}_{SF}$) and prove by a back-and-forth that the class satisfies all conditions on a quasi-minimal excellent class  except $\aleph_0$-homogeneity over models (see Definition~\ref{mcax}.2.c and Definition~\ref{qmdef}.1 and \ref{qmdef}.2.a; ii) Apply further deep geometric results to
obtain the needed homogeneity. Our first main theorem is

{

\begin{thmx}\label{A} All models of $T^{\infty}_{SF}$ are $(\infty,\omega)$ equivalent and satisfy the same Scott sentence. The associated first order theory $T_{j_{\Gamma}}$ is complete, admits elimination of quantifiers, and is $\omega$-stable.
\end{thmx}

Compared to the arithmetic setting, there are two important variants in the proof. The first variant is that the non-arithmeticity  both vastly simplifies the argument for Theorem~\ref{A} and allows the conclusion of $\omega$-stability at this point. This improvement allows the second, namely, a completely different argument for our next main result:

\begin{thmx}\label{B} The class is
$\aleph_0$-homogeneous over models, so almost quasiminimal- excellent and thus  categorical in all infinite
cardinalities.

\end{thmx}

Rather than relying on open-image theorems of Serre et al (\cite[\S 5.1]{DawHarris}), we apply the characterization \cite{BHHKK} of full excellence ($\aleph_0$-homogeneity over models) in terms of non-splitting, and thus we 
obtain  the result via $\omega$-stability.

The paper is organized as follows. After some background on Fuchsian groups \S~\ref{FG}, we describe in Section~\ref{formact} the formalism for the current study. In Section~\ref{infcompsec} we show Theorem~\ref{A} (Theorem~\ref{infcomp}) and introduce a notion of determined type which clarifies the argument for both the arithmetic and non-arithmetic cases.  In Section~\ref{aqme}, we use $\omega$-stability and the notion of determination to conclude Theorem~\ref{B} (Theorem~\ref{aqsmsuff} and Theorem~\ref{goal}).

\subsection*{Acknowledgements}
 We thank James Freitag, Anatoly Libgober, Andres Villaveces, as well as other participants of the Chicago-Bogota online seminar ``Conexi\'on de GALoiS'' for helpful conversations.
    }
\section{Background on Fuchsian groups and the covering maps}\label{FG}
We begin by recalling some of the basic facts about Fuchsian groups that will play a role in our analysis. Standard references for this section are \cite{Katok} and \cite{Shimura}.

Throughout, $\mathbb{H}$ will denote the upper half complex plane, while $\mathbb{H}^*:=\mathbb{H}\cup{\bf P}^1(\mathbb{R})$. For $F$ a subfield of $\mathbb{C}$, we will consider the action on $\mathbb{H}$ (and $\mathbb{H}^*$) of $SL_2(F)$ and $PSL_2(F)$ by linear fractional transformation: for 
$g=\begin{pmatrix}
    a & b \\
    c & d
  \end{pmatrix}
  \in SL_2(F)$ and $\tau\in\mathbb{H}$, we write
$g\cdot\tau=\begin{pmatrix}
    a & b \\
    c & d
  \end{pmatrix}\cdot\tau:=\frac{a\tau+b}{c\tau+d}$. In the case when $F=\mathbb{R}$, this action gives all the orientation preserving isometries of $\mathbb{H}$.
\par Let $\Gamma\subset PSL_2(\mathbb{R})$ be a Fuchsian group, namely, a discrete subgroup of $PSL_2(\mathbb{R})$. We assume throughout that $\Gamma$ is of the first kind and  of finite covolume. This means, respectively, that its limit set is ${\bf P}^1(\mathbb R)$ and that it has a fundamental domain with a finite area. Indeed, in this paper a ``Fuchsian group'' will always mean one where these two conditions hold. By a cusp of $\Gamma$, we mean a point $\tau\in\mathbb{H}^*$ such that its stabilizer group $\Gamma_{\tau}=\{g\in\Gamma\;:\;g\cdot \tau=\tau\}$ contains a parabolic element $\gamma\in\Gamma$, that is, an element $\gamma=\begin{pmatrix}
    a & b \\
    c & d
  \end{pmatrix}\in\Gamma$ with trace $\text{tr}(\gamma):=a+d=\pm 2$. It is a fact (cf. \cite[\S 1.3]{Shimura}) that $\Gamma$ acts on the set $C_{\Gamma}$ of its cusps and the action of $\Gamma$ on $\mathbb{H}_{\Gamma}:=\mathbb{H}\cup C_{\Gamma}$, yields a compact Riemann surface $\Gamma\setminus \mathbb{H}_{\Gamma}$ or equivalently a projective non-singular irreducible curve $S_{\Gamma}$. Notice that, if $\Gamma\setminus \mathbb{H}$ is non-compact, then the non-compactness occurs exactly at finitely many cusps. By an \emph{automorphic function} for $\Gamma$ (or $\Gamma$-automorphic), we mean a meromorphic function $f$ on $\mathbb{H}$ which is meromorphic at every cusp of $\Gamma$  and which is invariant under the action of $\Gamma$:
\[f(g\cdot \tau)=f(\tau)\;\;\;\text{ for all }\;g\in\Gamma\text{ and } \tau\in\mathbb{H}.\]
One usually denotes the field of $\Gamma$-automorphic functions by $\mathcal{A}_0(\Gamma)$ and the latter is isomorphic to the function field of the curve $S_{\Gamma}$. In particular, we have a distinguished automorphic function for $\Gamma$, denoted $j_{\Gamma}:\mathbb{H}_{\Gamma}\twoheadrightarrow S_{\Gamma}(\mathbb{C})$, called a uniformizer that is a generator of $\mathcal{A}_0(\Gamma)$. The function $j_{\Gamma}$ is not unique since there are nontrivial automorphisms of the curve $S_{\Gamma}$. Moreover, uniqueness follows once the value of $j_{\Gamma}$ at three points have been specified. The following notion is central  to the study of Fuchsian groups. 

\begin{definition}\label{commdef} Let $\Gamma$ and
	$\Gamma'$ be two Fuchsian groups.
\begin{enumerate}
 \item We say that $\Gamma$ and $\Gamma'$ are \emph{commensurable} (as subgroups of $PSL_2(\mathbb{R})$) if
$\Gamma \cap \Gamma'$ is of finite index in both $\Gamma$ and $\Gamma'$, written $\Gamma \sim \Gamma'$.

\item The \emph{commensurator} of $\Gamma$ is the subgroup of $PSL_2(\mathbb{R})$ defined as
    $$\text{Comm}(\Gamma)=\{g \in  {\psl}_2(\RR)\colon g^{-1} \Gamma g\ \text{\  is  commensurable with}\ \Gamma \}.$$
    \end{enumerate}
\end{definition}

In this paper we are mainly concerned with non-arithmetic Fuchsian groups. The definition of arithmeticity, while not difficult, is quite involved and will not be given here. We refer the reader to the references above. For the purposes of our work, it suffices and indeed essential to consider an equivalent characterization of arithmeticity due to Margulis.

\begin{fact}[Margulis] Let $\Gamma$ be a Fuchsian group. Then, $\Gamma$ is arithmetic if and only if it has infinite index in $\text{Comm}(\Gamma)$.
\end{fact}
In particular, a characteristic feature of arithmetic Fuchsian groups is the existence of infinitely many ``Hecke'' correspondences in $S_{\Gamma}(\mathbb{C})\times S_{\Gamma}(\mathbb{C})$. Let us make this notion of correspondences more precise; note that non-arithmetic Fuchsian groups will only have finitely many of those. If $g\in \text{Comm}(\Gamma)$ is given, then by definition we have three pairwise commensurable Fuchsian groups: $\Gamma$, $\Gamma^{g}:=g^{-1}\Gamma g$ and $\Gamma_g:=\Gamma\cap \Gamma^g$. It is not hard to prove that these groups have the same set of cusps (cf. \cite[Proposition 1.30]{Shimura}). So the uniformizers $j_{\Gamma}$ and $j_{\Gamma}\circ g$ of $\Gamma$ and $g^{-1}\Gamma g$ respectively, are also automorphic functions for their common subgroup $\Gamma_g$. A classical result due to Poincar\'e gives that any two automorphic functions for a Fuchsian group are related by a complex polynomial:

\begin{fact}\label{FactCorrespondence} For each $g\in  \text{Comm}(\Gamma)\setminus\Gamma$, there is a non-zero polynomial $\Phi_g\in\mathbb{C}[X,Y]$ such that for all $z\in \HH$,  $\Phi_g(j_{\Gamma}(z),j_{\Gamma}(g\cdot z))=0$ and $\Phi_g(j_{\Gamma}(z),y)=0$ if and only $y = j_{\Gamma}(h\cdot z)$ for some $h\in \Gamma g\Gamma$.
\end{fact} 

The polynomial $\Phi_g$ is called a {\em $\Gamma$-special polynomial} and
does not depend on the choice of $g$ but rather on the choice of the double coset $\Gamma g \Gamma$ \cite[Pages 51-52, 76-77]{Shimura}. A proof of its existence can be found in \cite[Page 185]{Lehner}. {Moreover, the full statement of Fact \ref{FactCorrespondence} is also obtained as a statement about finite branched covers of Riemann surfaces. Before giving the details of that approach, we fix some notation for non-arithmetic groups that we will use throughout the paper. Note however, that the argument will also hold for arithmetic groups since we will only use the properties of the commensurator of $\Gamma$.}

\begin{notation}\label{cosetdecomps}Let $G=\text{Comm}(\Gamma)$.
\begin{enumerate}
\item $[G:\Gamma] = k$ and $G = \cup_{i<k} \Gamma\cdot{\check g_i}$ where the $\check g_i$ (with $\check g_1=e$) are
a fixed set of coset representatives.

\item  $m:= m_g =[\Gamma:\Gamma_g]$  and $\Gamma=\cup_{i=1}^m\Gamma_g\cdot \epsilon_i$ with the $\epsilon_i$ coset representatives of $\Gamma_g$ in $\Gamma$.
\end{enumerate}
\end{notation}
Given the above notation, it follows that for each $g\in G$, the double coset $\Gamma g\Gamma$ has decomposition $\Gamma g\Gamma=\cup_{i=1}^m\Gamma \cdot g\epsilon_i$ (cf. \cite[Chapter 3]{Shimura}). Using the fact $m=[\Gamma:\Gamma_g]$ we have that the Riemann surface $\Gamma_g\setminus \mathbb{H}_{\Gamma_g}$ is an $m$-sheeted branched cover of $\Gamma\setminus \mathbb{H}_{\Gamma}$. In particular, the field $\mathcal{A}_0(\Gamma_g)$ of $\Gamma_g$-automorphic functions is a finite extension of degree $m$ of $\mathcal{A}_0(\Gamma)$ and since $j_{\Gamma}\in \mathcal{A}_0(\Gamma_g)$ and $j_{\Gamma}\circ g\in \mathcal{A}_0(\Gamma_g)$ we obtain the existence of $\Phi_g\in\mathbb{C}[X,Y]$. Furthermore, using the double coset decomposition $\Gamma g\Gamma=\cup_{i=1}^m\Gamma \cdot g\epsilon_i$, we see that the roots of $\Phi_g(j_{\Gamma}(z),Y)=0$ are precisely $j_{\Gamma}(g\epsilon_i z)$ for $i:1,\ldots,m$.

\begin{notation}\label{fixfield}
Throughout this paper, we fix a non-arithmetic Fuchsian group $\Gamma$ and write $G = \comm(\Gamma)$ for its commensurator.
We denote by $E_0$ the finitely generated (over $\mathbb{Q}$) subfield of $\mathbb{C}$ over which the curve $S_{\Gamma}$ is defined. We denote by $E$ the finitely generated (over $E_0$) subfield of $\mathbb{C}$ generated by the coefficients of those $\Gamma$-special polynomials \footnote{Technically, we should write $E_0(\Gamma)$ and $E(\Gamma)$, but there will be no ambiguity since the Fuchsian group $\Gamma$ is fixed.}. 
\end{notation}

In this paper $E$ will play the same role as the field $E^{ab}$, the maximal abelian extension of $E_0$, played in the papers 
\cite{DawHarris} and \cite{Etrev}.  In analogy to the arithmetic setting we distinguish between two kinds of points of the upper-half plane:

\begin{fact}\label{dichot}
{Fix $G = \comm(\Gamma)$.}  Each $a \in \mathbb{H}_{\Gamma}$ is either special or Hodge generic {for $\Gamma$}.
\begin{enumerate}
\item Hodge generic: No $g\in G$   fixes $a$.
\item Special: There is a $g\in G$ with exactly one or two fixed points in $\mathbb{H}_{\Gamma}$.  (This can be easily checked from the quadratic equation derived from $gz =z$ for $g\in G$.)

    Since $\Gamma$ has finite index in $G$ and is discrete, there are countably many special points and they are each in quadratic extensions of $\QQ$. Note that for every $g\in G$ either $\forall x f_g(x) \neq x$ or $\exists x f_g(x) = x$ is in the theory of $\langle H, {f_g:g \in G}\rangle$.
\end{enumerate}
\end{fact}

\begin{notation}\label{SpecialPoints}
Throughout we let $\Sigma$ denote the set of special points for $\Gamma$ and denote by $\text{Im}(\Sigma)$ the set $\{j_{\Gamma}(\sigma):\sigma\in\Sigma\}$. 
\end{notation}

\begin{remark}[The anatomy of special points]\label{spptdet}
Suppose $x\in \HH_\Gamma$ is special. Say, $x$ is one of (at most two points)  fixed by $g\in G =
\comm(\gamma)$. Then for any $h\in G$, $hgh^{-1}$ fixes $hx$ so $hx$ is also special (and has a partner just if $x$ does). If $j_\Gamma(x) = z$ then
for any $\gamma \in \Gamma$, $j_\Gamma(\gamma x) = z$. 
If $h\in G-\Gamma$, 
$j_\Gamma(h x)$ is algebraic over $E(z)$ by Fact~\ref{FactCorrespondence}. 
\end{remark}

In Definition~\ref{mcvocab} for the formal vocabulary  we will add constants for each special point
$x$ and  its image $j_\Gamma(x)$. This will avoid minor complications in the treatment
of special points when building a back-and-forth in Section~\ref{reduce}.

{\begin{remark}Clearly, given an arbitrary group $G$, commensurability is well defined for any pair of subgroups $H,H_1\subseteq G$. Below, we will also consider the more general definition of the commensurator of $H$ in $G$, namely $$\text{Comm}_G(H):=\{g \in  G\colon g^{-1}Hg\ \text{\  is  commensurable with}\ H \}.$$  \end{remark}}

First we go over some basic lemmas which are undoubtedly well-known. We include the simple {\em entirely group theoretic arguments} about a sequence of subgroups $H\subseteq \Gamma \subseteq G \subseteq \tilde G$.  In this paper $\psl_2(\RR)$ plays the role of $\tilde G$ and $\Gamma$ is a non-arithmetic Fuchsian group.

\begin{lemma}\label{whatsupergroup}  Consider a sequence of subgroups $H\subseteq \Gamma \subseteq G \subseteq \tilde G$.

\begin{enumerate}
\item If $[\Gamma: H] < \omega$ then $\comm_{\Gamma}(H) = \Gamma$. 

\item If  $G = \comm_{\tilde G}(\Gamma)$, $[G:\Gamma] < \omega$, and $[\Gamma: H] < \omega$ then $\comm_{
\tilde G}(H) = \comm_G(H) =G$.

Then $\Gamma$ and $H$ are commensurable in $\tilde G$, written $\Gamma \sim H$.                                                                                                       
\end{enumerate}
\end{lemma}
\begin{proof} Write $H_\delta$ for $H^\delta \cap H$.
\begin{enumerate} 
\item For every $\delta\in \Gamma$,   
 $[\Gamma :  H^\delta  \cap H] < \infty$   as  $H^\delta  \cap H$ is the intersection of two subgroups  of finite index in $\Gamma$.  
Even more easily,    $[H:  H^\delta  \cap H] < \infty$.  So each $H_{\delta} \sim \Gamma$ as required.                           

\item For 2, there are two cases:   
\begin{enumerate}[a)]
\item                                                                                                                                 
$\comm_{\tilde G}(H)\subseteq \comm_G(H)$: 
 Choose $\delta \in \tilde G$.
If $\delta \in \comm_{\tilde G}(H)$,  $H^\delta \sim H$. In particular, $[H: H_\delta] < \infty$. By conjugacy, $[\Gamma^\delta: H^\delta]< \omega$, so $H_\delta$ witnesses $\Gamma^\delta \sim H$. Since $\Gamma \sim H$, $\Gamma^\delta \sim \Gamma$
and $\delta \in \comm_{\tilde G}(\Gamma)= G$. That is, $\delta \in \comm_G(H)$.

\item  $\comm_G(H)\subseteq \comm_{\tilde G}(H)$ is obvious since $G\subseteq \tilde G$.

\end{enumerate}

\end{enumerate}
\end{proof}

Here is the crucial distinction arising from $[G:\Gamma]< \omega$. Lemma~\ref{finmancommens} establishes a $1$-$1$ correspondence
between the groups $H^{\delta}$ and the groups $H_\delta$. This correspondence holds as well in the arithmetic case.
But then there are infinitely many finite index subgroups.

\begin{lemma} \label{finmancommens} If $[G:\Gamma] < \omega$, there are only finitely many finite index subgroups of $\Gamma$ 
of the form $\Gamma_\delta = \delta \Gamma \delta^{-1} \cap \Gamma$ with $\delta \in \tilde G$.
\end{lemma}

\begin{proof} We argue for $\delta \in G$; the result extends to $\delta \in  \tilde G$ by Lemma~\ref{whatsupergroup}.

If for $i =1,2$, $\delta_i \in G-\Gamma$,  $\Gamma_i =\delta_i \Gamma \delta_i^{-1}\cap \Gamma$ and $\delta_1 \Gamma = \delta_2 \Gamma$ then $\Gamma_1 = \Gamma_2$. Since $[G:\Gamma]< \omega$ we are done.
\end{proof}

\begin{notation}\label{defcheck} Denote by $\check \Gamma$ the intersection of the finitely many finite index subgroups of the form $\Gamma_\delta=   \delta_ \Gamma \delta^{-1}\cap \Gamma$ with $\delta\in \tilde G$.
\end{notation}
The key point is that the argument in \cite{DawHarris,Etrev} relies only on the {finite index subgroups of $\Gamma$ of the form $ \delta \Gamma \delta^{-1} \cap \Gamma$ with $\delta \in G^{\rm ad}(\QQ^+)$.  Here, only $\delta\in G$ matter and so only finitely many  subgroups arise.}

Anatoly Libgober pointed out the following example to us.

\begin{example}[Non-arithmetic groups with infinitely many finite index subgroups]\label{finex}
{\rm Here is a construction of infinitely many non-arithmetic 
subgroups of $SL_2(\mathbb{Z})$. $SL_2(\mathbb{Z})$ contains a free subgroup (index 12) on 
2 generators. The commutator of the latter, say $\Gamma$ (i.e. the kernel of abelianization 
$F_{2} \twoheadrightarrow \mathbb{Z}^{2}$) has infinite index in $F_2$ and therefore also in $SL_2(Z)$ and 
hence is not arithmetic. So $\Gamma$ has finite index in $\comm_{\tilde G} (\Gamma)$. It has clearly infinitely many subgroups (all are free and non-arithmetic). 
$\Gamma$ is clearly a free group (and with some argument on infinitely many generators).  In particular any finitely generated finite group is  a quotient of $\Gamma$. The kernels of these surjections (even to finite $2$-generated groups) give infinitely many normal subgroups of finite index in $\Gamma$.  Nevertheless, by Lemma~\ref{finmancommens}, $\Gamma$ has only finitely many traces: $\Gamma_\delta$.}
\end{example}

\section{Formalization and Axiomatization}\label{formact}

We specify the vocabulary for our formal system.

\begin{notation} [The formal vocabulary $\tau$]\label{mcvocab}

 The two-sorted vocabulary $\tau$ consists of the sorts (unary predicate symbols)  $D$ the
     covering sort, $S$ the target sort, and  a  function $q$ mapping $D$ onto the sort $S$ and relations/functions
     on each sort.

 Fix a Fuchsian group $\Gamma$ with $G =Comm(\Gamma)$. Depending on the choice of uniformizer $j_{\Gamma}$, we get a standard model (Definition~\ref{proto}).

     \begin{enumerate}
     \item Domain sort:
 We write $\tau_G$ for the vocabulary of the first sort with a unary function symbol for each $g \in G =Comm(\Gamma)$. We use $f_g$ to name the functions
acting on $D$ to emphasize these of symbols in the vocabulary, not elements of models, but often write the shorter $g(x)$ or $gx$ instead of $f_g(x)$.

 \item  Variety/field Sort:
 The  vocabulary for the second sort, $\tau_S$, contains a set $\mathcal{R}$ of relations for each Zariski closed set of  $S(\CC)^n$ that is defined over $L$
and constants for $\Sigma$, and for the elements of $S(L)$, where  $L=E(\text{Im}(\Sigma))$ extends $E$ (Notation~\ref{fixfield}) by the set of images of the special points
$\Sigma$ described in Fact~\ref{dichot}.

\item $\tau $ is $\tau_S \cup \tau_G \cup \{J\}$.

 \item The covering map $J:D \twoheadrightarrow S$.

\end{enumerate}
\end{notation}
So an arbitrary $\tau$-structure will have the form:

$$M  = (D_M,J,S_M) := ((D_M, \{f_g:g \in G\}), J  ,(S(F_M),\mathcal{R}))$$ for some field $F_M$ extending\footnote{We write $D_M, S_M$ rather than the more model theoretically standard $D(M)$, $S(M)$ to avoid confusion with the algebraic notation $S(F_M)$.}\label{fnped} $L$. Notation~\ref{mcvocab} is our global vocabulary for studying covers of Fuchsian groups. However, in Notation~\ref{proto} we use a more evocative vocabulary to describe the standard model.  The axioms in Definition~\ref{mcax} are relative to an arbitrary Fuchsian group $\Gamma$; the new results of the paper require the assumption $[G:\Gamma]< \omega$ where $G = \comm(\Gamma)$;

Relying on the notations introduced after Fact~\ref{FactCorrespondence},  we describe the properties of the field sort more formally by describing a prototypical model by  a formal theory. We write {\em a}  prototypical variety because {\em a priori}, the structure depends on the choice of the covering map $\jbar$.

\begin{notation}[A prototypical Structure] \label{proto}

A standard model of the form $(D,J,S)$ is $$\mathbb{M}_{\jbar}  = ((\HH_{\Gamma}, \{f_g:g \in G\}),\jbar ,(S_{\Gamma}(\CC),\mathcal{R}))$$ where

\begin{enumerate}
\item The cover (domain) sort:

$ (\mathbb{H}_{\Gamma}, \{f_g:g \in G\} )$ the extended upper half plane $\mathbb{H}_{\Gamma}$ with the group $G= \comm(\Gamma)$    acting on it via fractional linear translations.
 
 \item The variety sort is ($S_{\Gamma}(\CC),\mathcal{R})$, with $\mathcal{R}$ as in Notation~\ref{mcvocab}.
 
 \item The map $\jbar$ is the uniformizer $j_{\Gamma}$ of $\Gamma$.

\end{enumerate}
We may occasionally write $H$ for $D$ or $S$ for $S_{\Gamma}$.

\end{notation}

We denote the first order theory of $\mathbb{M}_{\jbar}$ by $\th(\jbar)$. {We distinguish $\th(\jbar)$ from the first order axiomatization $T(\jbar)$ given in Definition~\ref{mcax}.1.  We prove in Theorem~\ref{infcomp} that $T(\jbar)$  implies $\th(\jbar)$.} Note that if one picks a different uniformizer $j'_{\Gamma}$, then one obtains a different theory $\th(\jbar')$ which is bi-intepretable with $\th(\jbar)$.

\subsection{Defining $Z_{\gbar}$}\label{defzg}

Relying on Notation~\ref{cosetdecomps},  we describe a key subset $Z_{\gbar}$ of $S^n$, that is clearly $\tau$-definable and show that it is $\tau_S$-definable (i.e., cut out by an algebraic subvariety of $S^n$). 

\begin{definition}\label{zgbar} 
Let $G=\comm(\Gamma)$ and $J: H \rightarrow S$ be the covering map. Consider a finite sequence of the form $\gbar = \langle e,g_2, \ldots, g_n\rangle$  from $G$ (by convention, $g_1=e$). We define the set $Z_{\gbar}$ to be the range of the map $J_{\gbar}: H \rightarrow S^k$ sending $x$ to $\langle  J(x),J(g_2 x),\ldots,J(g_k x)
    \rangle$.
\end{definition}
\begin{notation} Let $G$ and $J$ be as in the above definition.
\begin{enumerate}
\item  For a finite sequence $\gbar=\langle e,g_2, \ldots, g_n\rangle$, we sometimes abbreviate $\langle  J( x),J(g_2 x),\ldots,J(g_n x)
    \rangle$ by $J(\gbar x)$.    
\item In the finite index case, recall (see Notation \ref{cosetdecomps}) that $\check \gbar$ denotes a fixed enumeration of a set of $k$ coset representatives of $\Gamma$ in $G$. As above we write $J(\check \gbar x)$ and $Z_{\check{\gbar}}$ if the sequence is given according to this enumeration.
\end{enumerate}
\end{notation}

We now prove the $\tau_S$-definability by looking at looking the vanishing of the $\Gamma$-special polynomials on the prototypical model. The first order axioms in Definition~\ref{mcax}  transfer the results to the formal theory. 

\begin{definition}\label{Zgdef} Let $G=Comm(\Gamma)$ and $\gbar = \langle g_1=e,g_2, \ldots, g_n\rangle$ from $G$ be given. Define
     $${\Psi_{\gbar}:=}\left\{\Phi_{g_j g^{-1}_i}(x_i,x_j):1\leq i < j \leq n \right\},$$
where each $\Phi_{g_i g^{-1}_j}$ is the polynomial given in Fact \ref{FactCorrespondence}. Clearly $\Psi_{\gbar}$ is a finite set of polynomials from $E[x_1\ldots,x_n]$.
\end{definition}

The statement and arguments in the proof the next lemma (which is about the standard model) are well-known, especially in the case of the modular group, but we have not been able to find a specific reference. 

\begin{lemma}\label{Zgthm} For any $\gbar = \langle e,g_2, \ldots, g_n\rangle$ from $G=Comm(\Gamma)$, let $V(\Psi_{\gbar})$ be the algebraic subvariety of $S_{\Gamma}^n(\mathbb{C})$ defined by the polynomials in $\Psi_{\gbar}$. Then $Z_{\gbar}$ is the complex points of an irreducible component $Y_{\gbar}$ of $V(\Psi_{\gbar})$. In particular, we have that $Y_{\gbar}$ is an algebraic variety defined over $E$.
\end{lemma}
\begin{proof}
Let $\gbar = \langle e,g_2, \ldots, g_n\rangle$ be given from $G$. Consider the group $\Gamma_{\gbar}=\Gamma\cap(g_2^{-1}\Gamma g_2)\cap\cdots\cap(g_n^{-1}\Gamma g_n)$. Then from our choice of $\gbar$, we have that $\Gamma_{\gbar}$ is a finite index subgroup of $\Gamma$. Hence $\Gamma_{\gbar}$ is a Fuchsian group of the first kind with the same set of cusps as $\Gamma$. We consider $\Gamma\setminus \mathbb{H}_{\Gamma}$ and $\Gamma_{\gbar}\setminus \mathbb{H}_{\Gamma}$, the corresponding compact Riemann surfaces, as well as the corresponding projective non-singular curves $S_{\Gamma}$ and $S_{\gbar}:=S_{\Gamma_{\gbar}}$. We have a holomorphic map $f:\Gamma_{\gbar}\setminus \mathbb{H}_{\Gamma}\rightarrow (\Gamma\setminus \mathbb{H}_{\Gamma})^n$ defined as $f([z]_{\gbar})=(\jbar(z),\jbar(g_2z),\ldots,\jbar(g_nz))$. To see that $f$ is well-defined, that is does not depend on the choice of representative, simply observe that the map $\jbar_{\gbar}:\mathbb{H}_{\Gamma}\rightarrow \mathbb{C}^n$ is $\Gamma_{\gbar}$-invariant (see Definition~\ref{zgbar}).

Using the equivalence of categories between compact Riemann surfaces and smooth projective curves \cite[Proposition 1.95]{GG12}, we have an algebraic morphism $F:S_{\gbar}\rightarrow S_{\Gamma}^n$ whose analytification is $f$. Since $S_{\gbar}$ is projective, we have that $Y_{\gbar}:=F(S_{\gbar})$ is closed \cite[\S 5.2, Theorem 2]{Sha13}, and since it is irreducible, we see that $Y_{\gbar}$ is a $1$-dimensional irreducible closed subset of $S_{\Gamma}^n$. Furthermore, $Z_{\gbar}=f(S_{\gbar}(\mathbb{C}))=F(S_{\gbar})(\mathbb{C})=Y(\mathbb{C})$ and hence $Y_{\gbar}\subseteq  V(\Psi_{\gbar})$. Let $X$ be an irreducible component of $V(\Psi_{\gbar})$ containing $Y_{\gbar}$. Then $Y_{\gbar}\subseteq X$ and both are irreducible closed subsets of $S_{\Gamma}^n$. By maximality of irreducible components, it must follow that $Y_{\gbar}=X$. Hence $Y_{\gbar}$ is an irreducible component of $V$ and $Z_{\gbar}=Y_{\gbar}(\mathbb C)$.
\end{proof}

Observe that Daw-Harris and Eterovi\'c prove the definability of $Z_{\gbar}$ in \cite[Lemma 2.6]{DawHarris} and \cite[Lemma 3.22]{Etrev} using the Mumford-Tate module and refer to \cite{Milneshimura}. The proof of the field definability of $Y_{\gbar}$ appeared difficult in
\cite{Etrev} because it was cast as a trivial corollary of the proof that $[\phi_\gbar]: \Gamma_{\gbar}\setminus \HH \mapsto Z_{\gbar} $ is a morphism of Shimura varieties and so a special variety. Only the definability of the image is needed for the verifying the modular axioms in Definition~\ref{mcax}.

\subsection{The Axioms}
We now give the axioms for the infinitary theory: $T^{\infty}_{SF}(\jbar)$.
For any vocabulary $\sigma$, we denote by $L_{\omega_1,\omega}(\sigma)$ the collection of formulas obtained by inductively closing the atomic formula obtained by substitution of constants or variables in the relations/functions of $\sigma$ by 
countable Boolean connectives and quantification over finite sequences of variables.

\begin{definition}[Axioms]\label{mcax} 
\begin{enumerate}
\item $T(\jbar)$ denotes the following set of first order axioms (they are all true in the prototypical model $\mathbb{M}_{\jbar}$):

\begin{enumerate}
\item Each sentence in $\th(\langle\HH, \{f_g:g \in G\rangle)$. These
include `Special Point axioms' $SP_g$: For each  $g\in \Gamma$ that
fixes a unique point in $D$.

$$\forall x, y \in D
[(g(x) =x \wedge  g(y) =y) \Rightarrow x=y] $$
\item  $\th(S(\CC),\mathcal{R})$ (where $\mathcal{R}$ is as in Definition~\ref{mcvocab}).
\item{Connection axioms}
\begin{enumerate}
\item Equivariance axioms
\begin{enumerate}
\item For $\gamma \in \Gamma$,
$x =\gamma y  \rightarrow J(x) =J(y)$.
\item For $\gamma \in \Gamma$ and $\check g_i$
such that $g= \gamma g_i$

 $x =g y  \rightarrow J(x) =J(y)$
\end{enumerate}
\item Modular axioms\footnote{Lemma~\ref{Zgthm} shows the modular axioms hold in $\MMM_{\jbar}$. And so they can consistently be added here. Moreover,
when $[G:\Gamma] < \omega$ the axiom is needed only for $\check \gbar$.}: Let $\gbar=\langle e,g_2, \ldots, g_k\rangle$ be from $G$ and $Y_{\gbar}$ as given in Lemma~\ref{Zgthm}.
    \begin{enumerate}
\item $Mod^1_{\gbar}: \forall x \in D \ \{J(x),J(g_2 x), \ldots J(g_k x)\} \in Y_{\gbar}$.
\item  $Mod^2_{\gbar}: \forall z \in Y_{\gbar}\subseteq S^{k}  \ \exists x  \in D
    \left(J(x),J(g_2 x), \ldots J(g_k x)) =z \right) $ 
\item $${\bf MOD}  = \{ Mod^1_{\gbar}\wedge Mod^2_{\gbar}:
    \gbar\in G^{m}, m<\omega\}$$

\end{enumerate}
\end{enumerate}
\end{enumerate}
\item We add the infinitary axioms
\begin{enumerate}
\item $\Phi_\infty$ is the $L_{\omega_1,\omega}$ sentence asserting that  for $(D, S, J)$ both the transcendence degree of $S(F)$ over $\QQ$ 
and the strongly minimal  structure $\langle D, \{f_g: g \in \Gamma \}\rangle$ are infinite.

\item $SF$ (standard fibers) denotes the $L_{\omega_1,\omega}$-axiom:
$$(\forall x \forall y \in D(J(x)= J(y) \rightarrow \bigvee_{g \in \Gamma}
x =f_g(y)).$$

\item We extend  $T(\jbar)$ in two ways.
\begin{enumerate}
\item $T^\infty (\jbar)$ denotes $T(\jbar) \cup \{\Phi_{\infty} \}$ and

\item $T^\infty_{SF}(\jbar)$ denotes $T(\jbar) \cup \{SF\}\cup \{\Phi_{\infty} \}$.

\end{enumerate}
\end{enumerate}
\end{enumerate}
\end{definition}

When $[G:\Gamma]<\omega$, the standard fiber axiom can be rephrased as:
for each $e \in D$, there exists $i<k$ and $\gamma_e \in \Gamma$ such that
  $e = \check  g_i \gamma_e  d$ for some $\gamma_e \in \Gamma$.  This formulation
  clarifies some later arguments. 

\begin{remark}\label{ntypes}Observe \begin{enumerate}
  \item By Lemma~\ref{Zgthm}, in the theory $T(\jbar)$, we have that Axiom 1(c)ii. is true in $\MM_j$ and so consistent with the others. Thus, we can and will identify $Z_{\gbar}$ with $Y_{\gbar}$.

\item {\rm Since $Gd$ is the definable closure of any single element, the 
one type of any element controls its orbit. In particular there are only 
countably many $n$-types within the orbit for any $n$.}
\end{enumerate}
\end{remark}

\section{$L_{\infty,\omega}$-completeness}\label{infcompsec}

In this section we show that the infinitary
theory $T^\infty_{SF}(\jbar)$ is complete for $L_{\omega_1,\omega}$  (Definition~\ref{infcomdef})
and the induced complete first order theory $T(\jbar)$ is
$\omega$-stable with elimination of quantifiers.

\begin{definition}\label{infcomdef}
An $L_{\omega_1, \omega}$-sentence $\phi$ is {\em complete} if for every $L_{\omega_1, \omega}$-sentence $\psi$,
\medskip
$\phi\Rightarrow \psi$ or $\phi\Rightarrow \neg\psi$

\end{definition}

\begin{fact}[Karp's theorem]
$\phi$ is $L_{\omega_1,\omega}$-complete iff all 
models of $\phi$ are $(\infty,\omega)$-equivalent; 

i.e. back and
forth  equivalent by the quantifier free back-and-forth.
\end{fact}

One of the main goal of this section is to prove that any two models of $T^{\infty}_{SF}(\jbar)$ are $(\infty,\omega)$-equivalent.

\subsection{Reduction to the field type and consequences}\label{reduce}

Much of the argument here moves between $\tau$-structures
and field structures that are behind the scene but explicit.
Hence, we offer

\begin{warning}\label{warning} For a substructure $U$ of a model $M$ of  $T(\jbar)$ we think of
$S =U_S$ as a curve over a 
subfield $F_U$\footnote{$F_U$ extends the field $L$, which was defined in Notation~\ref{fixfield} by adding names for elements of $\text{Im}(\Sigma)$ to the field of definition $E$ of $S$.}.  A priori, a variety $V \subset S^n$ is defined by parameters $\abar$, which need not be in $S(F_U)$.  However, the stability of $ACF$ and the basic fact
for stable theories
that for any $A \subseteq M \models T$  defined by $\phi(\xbar,\bbar)$ in $M$ with $\bbar \in M$, there exists
 $\phi'(\xbar,\bbar')$    with $\bbar' \in A$ such that $\phi(\xbar,\bbar) \leftrightarrow \phi'(\xbar,\bbar') $  on $A$. We can apply this fact since $\phi'(\xbar,\ubar)$ is one of the relations in $\mathcal{R}$ over $L$. Thus, we write $S(F_U)$) without explicitly referring to this transformation (See Footnote~\ref{fnped} and Theorem~\ref{infcomp}.).
\end{warning}

We work with the following assumption in this section. We will conclude that the $\tau$-type over $A$ of any $d\in D_M-D_A$ is determined (Definition~\ref{defdet}) by $\qftp_{\tau_S}(J(\check\gbar(d))/A)$. 

 \begin{assumption}\label{context} {\rm Assume $k=[G:\Gamma]<\omega$. Abbreviate $\langle x, \check g_1 x, \ldots \check g_k x\rangle$ as $\check \gbar(x)$.
 Let  $M =\langle D,J ,S \rangle$ be a model of $T(\jbar)$.  Suppose $A$ is substructure of $M$ with sorts $D_A:= A\cap D$ and $S_A:=A\cap S$ such that $J$ maps $D_A$ onto $S_A$.
Moreover, $F_A$ is finitely generated over the field $L$, which was defined in Notation~\ref{fixfield}.}
\end{assumption}

The key use of the non-arithmeticity of $\Gamma$ is
 \begin{remark}\label{key} Since $J$ is constant on each coset of $\Gamma$ and $[G:\Gamma] = k$, for any $d \in D$ and any $g \in G$, we have that $J(gd)$ is in the finite set enumerated by $J(\check \gbar(d))$.
\end{remark}

Note that for any $\gbar$ from $G$, we have $ \gbar(d) \in \dcl(d) = Gd \cup \{J(\check\gbar(d))\}$ and so the quantifier-free $\tau$-type of $\gbar(d)$ over a set $X$ 
is implied by $\qftp_\tau (d/X)$.
Crucially, we will now show that
a stronger converse,  Lemma~\ref{3.3rev3}, holds. This requires a little more machinery.

\begin{notation}\label{reduct}{\rm Using the notation of Assumption~\ref{context}, we  compare here types in the vocabularies $\tau$
and $\tau_S$. In the proof of the next lemma we will look at $\tau$-formulas of the following special form: let $\sigma(v_1, \ldots, v_n)$ be a $\tau_S$-formula with parameters in $A$ and for any tuple $\gbar=(g_1,\ldots,g_n)$ from $G$, let $\sigma_{\gbar}(x)$ be the $\tau$-formula $\sigma(J(g_1 (x)), \ldots J(g_n (x)))$. Then clearly for any $d\in D_M - D_A$, 
\begin{equation} M\models \sigma_{\gbar}(d) \leftrightarrow  M\models \sigma (J(\gbar (d))). \label{eq1}\end{equation} Moreover and as noted above, since for any $g \in G$, the point $J(gd)$ is in the set enumerated by $J(\check \gbar(d))$, there is a $\tau_S$-formula $\theta(v_1,\ldots,v_k)$ so that the $\tau$-formulas $\theta_{\check\gbar}(x)$ and $\sigma_{\gbar}(x)$ coincide. For simplicity we will write $\check\theta$ instead of $\theta_{\check\gbar}$.

The following example may help. Suppose $A \subseteq M$ 
as in Definition~\ref{defdet} and $d\in D_M - D_A$.  Since $A$ is closed under $q$,  $J(\check \gbar (d) ) \cap A =\emptyset$. So there is no true instance of a  formula
$J(\check g_i d) = y$ in $\qftp_{\tau_S}(d/A)$ or $r=\qftp_{\tau_S}(J(\check \gbar d)/A)$. The formulas in $r$ are field formulas with free variables among $v_1, \ldots v_k$, where $k = [G:\Gamma]$.  The formula $\Phi_{\check g_j}(J(x), J(\check g_j x))=0$ (Fact~\ref{FactCorrespondence}) is 
in $\qftp_{\tau}(d/A)$ .
Whence, the formula $\Phi_{\check g_j}(v_1,v_j)=0$
 is in the $\qftp_{\tau_S}(\check \gbar (d)/A))= r$ and vice versa.
}
\end{notation}

\begin{definition}\label{defdet} Let $\tau_S \subset \tau $ be the vocabularies of $S(F)$ and $(D,J,S)$ 
respectively and $A   \subseteq M \models T(\jbar)$.
For $d\in D_M - D_A$ and $\bbar \in S_M - S_A$ we say
 $r=\qftp_{\tau_S} (J(\check \gbar(d) )\bbar)/A)$ {\em determines}
$\qftp_\tau(d\bbar/A)$ when for any $e\in D_M - D_A$,
if $J(\check \gbar e)$ realizes $r$, then $\qftp_\tau(d\bbar/A) = \qftp_\tau(e\bbar/A)$.
 \end{definition}

 Since $\Gamma$ is not arithmetic $\qftp_\tau(d\bbar/A)$ are $k$-types, because $J(Gd)$ is finite. In the 
 arithmetic case one needs an infinite sequence $g_id$, $g_i \in G$. A key hypothesis  (Assumption~\ref{context}) of the next lemma is that $J$ maps $D_A$ onto $S_A$. In Theorem~\ref{infcomp}, this will be easily
obtained while meeting requirement that the structure is finitely generated. {In the proof of Theorem~\ref{goal}, it will follow
from the hypothesis that $A = M$ is closed in the geometry.}

When we apply set theoretic operations to a sequence $\bbar$ we actually
mean to the set enumerated by $\bbar$.

\begin{lemma}\label{3.3rev3} 
Using the notation of Assumption~\ref{context}, for any $\bbar$ with 
$\bbar \cap J(\check \gbar(d)) = \emptyset$
and $\bbar\cap Gd = \emptyset$, we have that
$r=\qftp_{\tau_S} (\check \gbar(d)  \bbar/A)$ {\em determines}
$\qftp_\tau( d\bbar/A )$.
\end{lemma}

\begin{proof}
We   work  here with a singleton $d$; we extend to finite sequences $\dbar$ in remark~\ref{classify} and in the proof of Theorem~\ref{goal}. So, let  $d\in D_M-D_A$.
We first assume $\bbar$ is empty. Suppose  $e \in D_M$ and that
 $J(\check \gbar d)$ and $J(\check   \gbar e)$ realize the same quantifier free $\tau_S$-type $r$ over  $A$.

Note that the atomic formulas in one variable that might appear in $\qftp_{\tau}(d/A)$
and do not involve field formulas are of two kinds. i)  $J(x) = a$ for some $a\in A$; all such are false 
because $A$ is closed under $J$.  ii) $g(x) =x$; all are false because all $d \in D_M -D(L)$ are 
Hodge generic. {This says also that no such atomic formula can differentiate elements of $D_M-D_A$. Indeed they will all satisfy the negations of these formulas.}

As explained in Notation~\ref{reduct}, it remains to consider formulas of the form $\check \theta(x)$ in $r_1 =\qftp_\tau(d/A)$ and $r_2=\qftp_\tau(e/A)$.  If $\check \theta(x)\in r_1 $, using the equivalence~(\ref{eq1}) we get that $\theta(\vbar)\in r=\qftp_{\tau_S}(J(\check \gbar d)/A )$. Our assumption that $J(\check   \gbar e)$ also realizes $r$ gives, using the equivalence~(\ref{eq1}) once more, that $\check \theta(x)\in r_2$. Reversing the roles of $J(\check \gbar d)$ and $J(\check   \gbar e)$ gives that $r_1=r_2$

Now assume that $\bbar \neq \emptyset$. Recall that we assume $J(\check \gbar (d)) \cap \bbar = \emptyset$. So there are no positive formulas from $\tau -\tau_S$
in $\qftp_\tau(J(\check \gbar d)\bbar/A ) - \qftp_{\tau_S}(J(\check \gbar d)\bbar/A )$. Since $\qtp_{\tau_S} (\check \gbar(d) )\bbar/A)$ implies $\qtp_{\tau_S} (\check \gbar(d) )/A)$, which implies $\qtp_{\tau_S} (\check \gbar(d) )/A)$, we finish.
\end{proof}

\subsection{The back and forth}

In the proof here
and in \cite{DawHarris, Etrev} we actually work in an extension given by taking 
the substructure $U,U'$ to be given by  finitely generated (over $L$)   fields $L_U$, $L_{U'}$. While the finitely generated steps in the argument
specialize to $S(F)$ only via  Warning~\ref{warning}, the final back and forth is immediate: A back and forth between two structures immediately determines one on any relativized reduct.

\begin{definition}[The back and forth scheme]\label{ps} {\rm
Fix two models $M =(D,J,S(F))$ and $M' =(D',J',S(F'))$
 of $T(\jbar)$. Note here, that $F:=F_U$ and $F':F_{U'}$. Consider the set
of finitely generated substructures (some of the generators may be algebraically independent) of $M$ and $M'$ and the collection
$I$ of partial isomorphisms such that:
For each $f \in I$, $\dom f$ and ${\rg \  f}$ are each finitely generated
over $L$ as in Assumption~\ref{context}.
 A typical member $f$ of the system  $I$ has
 $\dom f = U = D_U \cup
S_U$. Since $U$ is finitely generated, $D_U$ consists of the $G$-orbits of a
finite number of $x\in D$; $S_U$ is $S(L_U)$
 where $L_U$ 
is the field generated\footnote{Note that the additional points obtained by including the orbits
determine only finitely many new field elements since $J$ is constant on each
orbit, so the field remains finitely
 generated.} by $L$ (since the elements of
$L$   are named), but now extended by  the coordinates of the $J(x)$ for $x \in
D_U$. Define a
similar subsystem for $M'$,  labeling by putting  primes on corresponding
objects.
}
 \end{definition}
Our task is to prove this scheme has the back-and-forth property. Note that by definition, every point of $D$ is either Hodge generic or special (and thus
named in the vocabulary). So we can and will ignore the special points in building the back and forth system.

\begin{notation}\label{ftype}  For a type $r(v)$ over a set $A$ and an isomorphism $f$ from $A$ to $B$, $f(r)$
is the set of $B$-formulas $\phi(v,f(\abar))$ with $\phi(v,\abar)\in r$.
\end{notation}

Since our main goal concerns uncountable models we must deal with models of infinite transcendence degree, which is preserved by  $L_{\infty,\omega}$ equivalence. Thus,
we restrict to that case.

\begin{theorem} \label{infcomp} Suppose  that $M$ and $M'$ satisfy $T^{\infty}_{SF}(\jbar)$. Then, the
    $qf$-system  defined above  is
a quantifier-free {\em back and forth} system between them.  
There is then a unique
countable model $M$  of $T^{\infty}_{SF}(p)$, and hence $T^{\infty}_{SF} (p)$ is axiomatized by a complete sentence of $L_{\omega_1,\omega}$, the Scott sentence\footnote{Any countable structure is characterized by a sentence of
$L_{\omega_1,\omega}$, its Scott sentence.} of $M$. Furthermore, the first order theory $T(\jbar)$ is complete and
admits elimination of quantifiers. 
 \end{theorem}

\begin{proof}  Our back and forth scheme has an isomorphism $f$  between $U \subseteq M$
and $U' \subseteq M'$. Note $f$ restricts to a $G$-equivariant
injection of $D_U$ into $D'_U$ and a
bijection between $S(L_U) \subseteq S(F)$ and $S(L_{U'}) \subseteq S(F') $
that fixes $L=E(\Sigma)$. Here we write $F, F'$ for $F_U,F_{U'}$ respectively.

For $x \in M - U$, we must find $x' \in M'$ so that $f \cup \{\langle x,x'\rangle\}$ generates
an isomorphism between the structures generated by $U \cup \{x\}$ and $U'
\cup \{x'\}$.
If $x\in S(F)$, $x = J(d)$ for some $ d\in D = D_F$, finding $d'$ which generates an isomorphism between the structures generated by $U \cup \{d\}$ and $U'
\cup \{d'\}$ containing $x'=J(d')$ suffices. So we reduce to
the  case $x \in D-D_U$. Note, however, that since $J$ maps onto $S(F)$ we eventually
have considered every finite sequence from $S(F)$.

So let $x \in D- D_U$
(and so 
Hodge generic). As in Lemma~\ref{3.3rev3}, let $r = \qftp_{\tau_S}(J(\check \gbar(x))/U)$.
 Since the formulas in $r$ are over $L_U$ and $f$ is an isomorphism between $L_U$
 and
 $L_{U'}$,  $f(r)$ is a consistent type over the finite set of generators of  $L_{U'}$.
Let $W_{r}$ be the minimal subvariety of $S(F)^n$ containing $J'(\check \gbar(x))$ and defined over $L_U$.
By  $Mod^1_{\gbar}$ and minimality, since $J(\check \gbar(x))$ is contained in $Z_{\check \gbar}(F)$, we have that $W_{r}$ is  contained in
$Z_{\check \gbar}(F)$. We need to find $x' \in D'$ such that $J(\check \gbar(x'))$ satisfying $f(r)$.
 Since
 $Z_{\check \gbar}(F)$  is defined over $L$,
 $f (Z_{\check \gbar}(F))=Z_{\check \gbar}(F')$ and 
  $$f(W_{U,d}) \subseteq f (Z_{\check \gbar}(F))=Z_{\check \gbar}(F').$$

Now over $L_{U'}$, the type $f(r)$ determines a minimal variety $W'_{f(r)}$ which satisfies $W'_{f(r)} \subseteq  f(W_{U,d})$. Given our assumption on the transcendence degree of $F'$ (recall $F' \models T^{\infty}_{SF}(\jbar)$), we can find an $L_{U'}$-generic point $h$ of $W'_{f(r)}$ in $Z_{\check \gbar}(F')$. Now by $Mod^2_{\gbar}$
there is an $x'$ in $D'_U$ with $J'(\gbar(x')) = h$. 
Now using Lemma~\ref{3.3rev3} we have that $r=\qftp_{\tau_S} (\check \gbar(x)  / L_U)$  determines
$\qftp_\tau( x  / L_U )$ and similarly $f(r)$ determines $\qftp_\tau( x'  / L_{U'} )$. Using this we get that $f \cup \{\langle x,x'\rangle\}$ generates
an isomorphism between the structures generated by $U \cup \{x\}$ and $U'
\cup \{x'\}$ which is what we aimed to prove. 

Thus all models of $T^{\infty}_{SF} (\jbar)$ are $({\infty,\omega})$-equivalent.
In particular, the standard model, $M_{\overline{\QQ^*}}$ over the countable $\omega$-saturated ACF, $\QQ^*$, is the unique countable model of the 
(necessarily) complete for $L_{\omega_1,\omega}$-theory $T^{\infty}_{SF}(\jbar)$ \cite[\S 6]{Baldwincatmon}.

For the furthermore, 
by compactness every model $M$  of $T (\jbar)^{\infty}_{SF}$ has an $\omega$-saturated 
first order elementary extension $M_1$ satisfying  $T^{\infty}(\jbar)$.
This model will fail SF since for any $d\in D$,   
$p^{-1}(p(d))$ will contain infinitely many elements that are not in the $\Gamma$-orbit
of $d$. But, the restriction $M_2$ of each $p^{-1}(p(d))$ to the $\Gamma$-orbit of $d$ is an elementary extension of $M$  
satisfying $T^{\infty}_{SF}(\jbar)$. And the back and forth argument works as well for the class of $\omega$-saturated models.
 Thus,  $T(\jbar)$ admits elimination of quantifiers and is complete \cite[Proposition 29]{Pilnotes}. 
 \end{proof}

 By the completeness of $T(\jbar)$, we can identify it with $\th(\jbar)$.

\begin{corollary}\label{wstab} The first order theory $T(\jbar)$ is $\omega$-stable.
\end{corollary}

\begin{proof} By Lemma~\ref{3.3rev3}, for any model of $T(\jbar)$,
each quantifier-free type $\tau$-type of an element $d$ over  $D_M$ is determined by its restriction to
$\qtp_{\tau_S}(\tilde \gbar(d)/S_M)$.   Since the theory of $S_M$ is $\omega$-stable, there are only $|M|$
such types.
\end{proof}

Note that $\omega$-stability cannot  be deduced at this stage by the identical argument in the arithmetic case, because in that case each quantifier-free type of an element $d$ in $D$ is determined 
by the type of an infinite sequence from the variety $S(F_M)$.
  
\section{Almost Quasiminimal Excellence}\label{aqme}
As pointed out in \cite{BaldwinVill}, the proof of categoricity of covers in dimensions greater than 
one requires the use of otop-methods.  However, for
curves a minor variant on the approach in \cite{BaysZil, BHHKK}  suffices.

\begin{definition}
[Quasiminimal excellent geometries]\label{qmdef}
 Let $\bK$ be a class of $L$-structures such that $M \in \bK$ admits a
 closure relation $\cl_M$  mapping $X \subseteq M$ to $\cl_M(X) \subseteq M$
  that satisfies the following properties. We write $X \leq M$ if $\cl_M(X) = X$.

\smallskip
\begin{enumerate}
\item \textbf{Basic Conditions}
\begin{enumerate}
\item
 Each $\cl_M$ defines a pregeometry on $M$.
\item For each $X\subseteq M$, $\cl_M(X) \in \bK$.

 \item countable closure property (ccp): If $|X| \leq \aleph_0$ then  $|\cl(X)| \leq \aleph_0$.

\end{enumerate}

\item \textbf{Homogeneity}

\begin{enumerate}
\item 
A class $\bK$ of models has {\bf $\aleph_0$-homogeneity over $\emptyset$}  if the  models of $\bK$
are pairwise qf-back and forth equivalent (Definition~\ref{ps})

\item
A class $\bK$ of models has \textbf{$\aleph_0$-homogeneity over models}
if for any $G
 \in \bK$ with $G$ empty or a countable member of
$\bK$, any $H,H'$ with $G\leq H, G\leq H'$,  $H$ is qf-back and forth equivalent with $H'$ over $G$.
\end{enumerate}
\item $\bK$ is an {\em almost quasiminimal } if the universe of any model $H\in \bK$ is in $\cl(X)$ for any maximal $\cl$-independent set $X \subseteq H$.
\item We call a class   which satisfies  conditions 1) through 3)
an {\em almost quasiminimal excellent class (of geometries) } \cite{BHHKK}.
\end{enumerate}
\end{definition}  

`Almost' appears because in our case there are two interalgebraic `generic' types.
The key result \cite{Zilbercatex, BHHKK, BaldwinVill} is:

\begin{theorem}\label{aqsmsuff} If $\bK$ is a (class of) almost quasi-excellent geometries, then $\bK$ is categorical in all uncountable cardinalities.
\end{theorem}

 \begin{definition}\label{cldef} For $(D,J, S(F)) \models T^\infty_{SF}(\jbar)$ and $X$ a substructure (recall $D_X = X\cap D$ and $S_X =X \cap S(F)$) we define,

$$\cl(X) = \acl( q^{-1}(\acl(q(
D_X))))  \cup \acl(S_X )$$
 where $\acl$  is the  algebraic
closure in the normal model theoretic sense.
\end{definition}

We need one further step, Theorem~\ref{goal}, to establish that the models of $T^\infty_{SF}(\jbar)$ are almost quasiminimal excellent.
The crux is to apply
Fact~\ref{sisosuff};  that requires a few preliminaries. Recall that (by definition) $p=\tp(\abar/M)$ splits over $A\subset M$
if there are $\bbar_1, \bbar_2 \in M$ realizing the same type over $A$ and a formula $\phi(\xbar,\ybar)$
with $\phi(\xbar,\bbar_1)\wedge \neg\phi(\xbar,\bbar_2) \in p$.

\begin{fact} \cite[Cor 5.5]{BHHKK} \label{sisosuff} If $\bK$ satisfies all conditions for `excellence' except
 $\aleph_0$-homogeneity over models,
the following are equivalent.
\begin{enumerate}
\item  If $M\leq_{\cl} N \in \bK$, $M$ is countable, 
and $\abar\in N- M$ then there is a finite subset $X$ of $M$ such that $\tp(\abar/M)$ does not split over $X$.
\item $\bK$ satisfies $\aleph_0$-homogeneity over models.
\end{enumerate}
\end{fact}

We need one observation to apply this method. Since an irreducible curve is strongly minimal we have:

\begin{fact}\label{sm} 
The induced first order theory on any irreducible curve is strongly minimal and so
categorical in all infinite cardinalities.
\end{fact}

 \begin{remark}\label{auto}[Automorphisms of single orbits] Note that for any
 $d\in D(\mathbb{M})$, its orbit $Gd$ is in the definable closure of $d$ (by the unary functions). But the automorphism groups of $\Gamma$ acts naturally on $\check g_i \Gamma$ for each $i\leq k$.  Thus, for any $M\leq \mathbb{M}$, $\aut(Gd/M) \cong Aut(\Gamma) \bigoplus  G/\Gamma$.
 \end{remark}                

Recall that in \cite{BHHKK} types are quantifier free types and that we established quantifier elimination for $T(\jbar)$ in
Theorem~\ref{infcomp}. The following remark will be used in the proof of our next theorem.

\begin{remark}\label{classify} Let $M$ and $N$ be models of $T^{\infty}_{SF}(\jbar)$ such that $M\leq_{\cl} N$.
Two elements $d,e \in N$ are related in one of the following ways. If $e$ is not in the $G$ orbit of $d$ then $r=\qtp_{\tau_S}(d\abar/M)$ {\em determines}
$\qtp_\tau (d \abar/M)$, because the action of $G$ cannot introduce any new information.
There are several possibilities if $e = gd$ for $g\in G$. If 
$g\in \Gamma$, of course, $\jbar(d) =\jbar(e)$. If $g \in \Gamma f \Gamma$, the $\Phi_f(\jbar(d),\jbar(e))= 0$ and, even more if $g$ and $e$ are in the same coset of $\Gamma_k$ in $G$,
then  $\jbar(d) =\jbar(e)$.
\end{remark}

The categoricity result is immediate from Theorem~\ref{aqsmsuff} and the following.

\begin{theorem}\label{goal} The models of $T^{\infty}_{SF}(\jbar)$, with $\cl$ defined as in \ref{cldef}, form an almost  quasiminimal excellent class.
\end{theorem}

\begin{proof} 
By \ref{3.3rev3} and \ref{sisosuff}: $1) \rightarrow 2)$, we need only show that for models $M\leq_{\cl} N$,  of $T^{\infty}_{SF}(\jbar)$ and
  any  finite tuple $\ebar\in N-M$  there is a finite subset $X$ of $M$ such that $\tp_\tau(\ebar/M)$ does not split over $X$. We rely on the reduction of types to the field sort in Section~\ref{reduce}. For any $d \in D_N$, write $\check \gbar(d)$ for $\langle \check g_1(d) \ldots \check g_k(d)\rangle$.

We modify $\ebar$ to 
 $\abar \in (N - M)$, written as $\abar_1\abar_2\abar_3$  as follows. Let $\abar'_1 \in D_N- D_M$ 
contain one $d_i$ for each of the, say, $r'$ orbits in  $D_N-D_M$ that intersect $\ebar$.

 For this we choose one component of $\ebar$ that is in $D_N $ for each
$G$-orbit $[e_i]$ that intersects $\ebar$ to form $\abar_1$.
We can make the restriction since any other element in $[e_i]$ is described as $g e_i $ for some $g\in G$.
Choose
 $\abar_2 \in S_N-S_M$  as  the tuple $\langle J(\check \gbar(d_i)):i<r\rangle$, perhaps extending $\ebar$ by adding images of some points in $\abar_1$.
 Finally, let
$\abar_3$ be    $((S_N-S_M) \cap \ebar)-\abar_2$ (with length $m$).

Observe that for any sequences $\fbar$ and $\hbar$  in $N-M$
and for any theory if $\qftp(\fbar/M)$ splits over a finite $X\subset M$, then
$\qftp(\fbar\hbar/M)$   
splits over $X$ by the same witness. Note that, by
Fact~\ref{FactCorrespondence},
there are no special
polynomial definable relations among the components of $\abar_2$ although there are on
each $J(\check \gbar(d_i))$. 

The key to the rest of the proof is that by the monotonicity under extensions of $\abar_2$ in Lemma~\ref{3.3rev3} and  since $\check \gbar(d_i) \subseteq \abar_2$,
we know
$\qftp_{\tau_S} (\abar_2\abar_3/M)$ determines $\qftp_{\tau}(d_i/M)$ for each $i$. We will show this is contradicted if $\qftp_\tau(\abar/M)$ splits over some finite
$X\subset M$.

Note
first that   $\qftp_\tau(\abar_1/M)$   cannot split
over any finite subset of $D_M$ because  each element of
$\abar_1$   has no relation to any element
of $D_M$ (as $M$ is a union of orbits) and there are only unary operations on $D_N$.   Similarly, each element of
$\abar_2\abar_3$ has no quantifier free relation with any element of $D_M$ since $M$
is closed under $J$. {\em Thus any witness to splitting must occur
in $S_M$.}

 By $\omega$-stability of $S_M$, choose finite $X$ such that $\qftp_{\tau_S}(\abar_2\abar_3/M)$ does not split over $X$. 
 We claim this non-splitting extends to $p=\qftp_{\tau}(\abar_1\abar_2\abar_3/M)$ not splitting over the same $X$.

 By Lemma~\ref{3.3rev3},
$\qftp_{\tau_S}(\abar_2\abar_3/M)$ determines $\qftp_{\tau}(d_i\abar_2\abar_3/M)$ for $1\leq i \leq r$. 
Moreover, since the $d_i$ are in distinct orbits, for all $g\in G$ and any $i,j$, $g d_i \neq d_j$, Further, $J(g d_i) = J(\check g_j(d_i))$ iff $g\in \check g_i \Gamma$.
Thus, $\qftp_{\tau_S}(\abar_2\abar_3/M)$ determines $\qftp_{\tau}(\abar_1\abar_2\abar_3/M)$; we now deduce that $ p=\qftp_{\tau}(\abar_1\abar_2\abar_3/M)$ does not split over $X$.  
 Write the variables of $p$ as $\ubar\vbar\wbar$, $\ubar$ has length $r$, $\vbar$ has length $r\times k$, indexed as
 $v_{i,j}$ for $1 \leq i \leq r$ and  $1 \leq j \leq k$, and $\wbar$ has length $m$. Recall $k=[G:\Gamma]$.

If $p=\qftp_\tau(\abar/M)$
  splits over $X$ there must be $\bbar_1$ and $ \bbar_2$ realizing the same type over $X$
and a formula $\theta(\ubar,\vbar,\wbar,\zbar)$ that
each of $\theta(\abar,\bbar_1 ) $ and $ \neg \theta(\abar,\bbar_2 ) $  satisfy $p$.
Since  $\qftp_{\tau_S}(\abar_2\abar_3/M)$ does not split over $X$, some $u_i$  must occur in $\theta$. More precisely, a formula
$J(\check g_i u_{\ell} ) = v_{\ell,i}$ must occur positively in $\theta$ for at least one $i\leq k$ and $\ell \leq r$.

By Lemma~\ref{sm}, $S_M$ is strongly minimal 
so finitely homogeneous \cite{BaldwinLachlansm}.
Thus, there is an automorphism $\alpha$
of $S_M$ fixing $X$ pointwise and taking $\bbar_1$ to $ \bbar_2$ and (by non-splitting)
fixing $\qftp_{\tau_S}(\abar_2\abar_3/M)$, and thus 
$\qftp_{\tau}(\abar_2\abar_3/M)$.  As, the only (neg) atomic formulas in $\tau -\tau_S$ with arguments
in both $\abar_1$ and $\abar_2\abar_3$
 assert for each $f\in D_M$, $J(f) \neq a_{2,j}$  and $J(f) \neq a_{3,\ell}$ because $M$ is closed under $q$. 

If $\theta$ gives opposite truth values depending on the choice of $
\bbar_i$, $p$ and $\alpha(p)$ are distinct types over $M$ that agree on $M\abar_2\abar_3$.
This contradicts our observation earlier in this proof that Lemma~\ref{3.3rev3} implies  $\qftp_{\tau_S}(\abar_2\abar_3/M)$ determines $\qftp_{\tau}(\abar_1\abar_2\abar_3/M)$. 
This completes the argument for Theorem~\ref{goal}.\end{proof}


\newcommand{\etalchar}[1]{$^{#1}$}

\end{document}